\documentclass[12pt,psamsfonts]{amsart}
\usepackage{amsmath}
\usepackage{amsthm}
\usepackage{amssymb}
\usepackage{amscd}
\usepackage{amsfonts}
\usepackage{amsbsy}
\usepackage{graphicx}
\usepackage[dvips]{psfrag}
\usepackage{array}
\usepackage{color}
\usepackage{epsfig}
\usepackage{url}
\usepackage{overpic}
\usepackage{tikz-cd}
\usepackage{enumitem}
\usepackage{hyperref}
\usepackage{soul}
\usepackage{float}
\usepackage[normalem]{ulem}

\usepackage{textcomp}

\newcolumntype{L}{>{\displaystyle}l}
\newcolumntype{C}{>{\displaystyle}c}
\newcolumntype{R}{>{\displaystyle}r}

\def\t{\theta}

\newcommand{\R}{\ensuremath{\mathbb{R}}}
\newcommand{\N}{\ensuremath{\mathbb{N}}}

\newcommand{\CF}{\ensuremath{\mathcal{F}}}

\newcommand{\CO}{\ensuremath{\mathcal{O}}}
\newcommand{\ov}{\overline}
\newcommand{\la}{\lambda}

\newcommand{\T}{\theta}
\newcommand{\f}{\varphi}
\newcommand{\al}{\alpha}

\renewcommand{\L}{\Lambda}

\renewcommand{\O}{\Omega}
\newcommand{\dom}{\mathrm{dom~}}
\renewcommand{\Im}{{\rm Im~}}
\renewcommand{\ker}{{\rm Ker~}}
\newcommand{\coker}{{\rm Coker~}}

\renewcommand{\d}{\mathrm{d}}
\newcommand{\Id}{\mathrm{Id}}

\newcommand{\s}{\ensuremath{\mathbb{S}}}

\newcommand{\de}{\delta}

\newcommand{\wt}{\widetilde}

\DeclareMathOperator{\dist}{dist}

\def\p{\partial}
\def\e{\varepsilon}

\newtheorem {theorem} {Theorem}
\newtheorem {definition} {Definition}
\newtheorem {proposition}{Proposition}
\newtheorem {corollary}{Corollary}
\newtheorem {lemma}{Lemma}

\newtheorem {remark}{Remark}

\newtheorem {mtheorem} {Theorem}

\newcommand{\ds}{\displaystyle}

\usepackage{mathpazo}

\begin{document}
\renewcommand{\arraystretch}{1.5}

\title[Averaging for continuous differential equations]
{Higher order analysis on the existence of\\ periodic solutions in continuous differential equations via degree theory}

\author[D.D. Novaes and F. B. G. Silva]
{Douglas D. Novaes and Francisco B. G. Silva}

\address{Departamento de Matem\'{a}tica, IMECC, Universidade
Estadual de Campinas, Rua S\'{e}rgio Buarque de Holanda, 651, Cidade
Universit\'{a}ria Zeferino Vaz, 13083--859, Campinas, SP, Brazil}
\email{ddnovaes@unicamp.br, fbruno@ime.unicamp.br}



\begin{abstract}
Recently, the higher order averaging method for studying periodic solutions of both Lipschitz differential equations and discontinuous piecewise smooth differential equations was developed in terms of the Brouwer degree theory. Between the Lipschitz and the discontinuous piecewise smooth differential equations, there is a huge class of differential equations lacking in a higher order analysis on the existence of periodic solutions, namely the class of continuous non-Lipschitz differential equations. In this paper, based on the degree theory for operator equations, we perform a higher order analysis of continuous perturbed differential equations and derive sufficient conditions for the existence and uniform convergence of periodic solutions for such systems. We apply our results to study continuous non-Lipschitz higher order perturbations of a harmonic oscillator.
\end{abstract}

\maketitle

\section{Introduction and statements of the main results}
 The Averaging Method is a classical tool which is concerned with providing asymptotic estimates for solutions of non-autonomous differential equations in the following standard form \begin{equation}\label{eq:e1}
	x^\prime = \e F(t,x,\e),
	\end{equation}
	where $F\colon\R\times D\times[0,\e_0]\to \R^n$ is a continuous function $T-$periodic in the variable $t,$ with $D$ being an open subset of $\R^n$ and $\e_0 > 0.$ Such asymptotic estimates are given in terms of solutions of an ``averaged equation.'' 
	
The averaging method dates back to the works of Lagrange and Laplace, who provided an intuitive justification of the process and applied it to the problem of perturbations in the solar system \cite{sanders2007averaging}. The first formalization of this procedure was given by Fatou in 1928 \cite{Fa}.  Important contributions to the theory were made by Krylov and Bogoliubov \cite{BK} in 1934 and Bogoliubov \cite{Bo} in 1945. From then on, it has been an efective tool in studying qualitative properties of ordinary differential equations.

In particular, the averaging method has been proven to be very useful in detecting periodic solutions. In \cite{Hale,sanders2007averaging,Ver06}, one can find results providing sufficient conditions for the existence of periodic solutions for sufficiently smooth differential equations. These results are based in a first order analysis, which means that the obtained sufficient conditions depend only on the function $(t,x)\mapsto F(t,x,0),$ neglecting all the information that  the derivatives with respect to $\e$ of $F(t,x,\e)$ can provide. Later, \cite{buicua2004averaging} extended these results for studying the existence and also convergence of periodic solutions for continuous differential equations, Lipschitz or not. The existence of periodic solutions by the first order averaging method can also be obtained as an immediate consequence of the continuation result \cite[Theorem IV.1]{gaines1977coincidence}.  More recently, \cite{LliNovTei2014} performed a higher order analysis for studying periodic solutions of Lipschitz-continuous differential equations in terms of Brouwer degree theory. The averaging method has also been extended for non-smooth differential equations. In this context, the studies by \cite{ItiLliNov17, LliNovRod17, LlilNovTei15, LliMerNov15} generalize the averaging method at any order for studying periodic solutions of discontinuous piecewise smooth differential equations.
	
Between the Lipschitz-continuous and the discontinuous piecewise smooth differential equations, there is a huge class of differential equations \eqref{eq:e1} lacking in a higher order analysis on the existence of periodic solutions, namely the class of continuous (non-Lipschitz) differential equations.  Continuous differential equations with non-Lipschitz nonlinearities appear naturally in applications. We may quote, for instance, neural networks \cite{Bian12,Forti06,Huaiqin11}, weather and climate models \cite{Hanke20,Porz18}, incompressible fluid dynamics \cite{Frisch95,Kol91}, and biological models of competition \cite{Zuowei13}. Under the Lipschitz assumption, analysis strongly relies on the uniqueness property enjoyed by the solutions of differential equations. However, in general, this property is lost for continuous differential equations. Here, motivated by the analysis performed in \cite{buicua2004averaging}, we take advantage of the degree theory for operator equations to overcome this difficulty and perform a higher order analysis on the existence of periodic solutions for continuous differential equation in the standard form \eqref{eq:e1}.

As a fundamental hypothesis on differential equation \eqref{eq:e1}, we shall assume that for a given open bounded subset $V\subset \R^n,$ with $\ov V\subset D,$ 
\begin{itemize}
\item[${\bf H.}$] there exists $\e_1\in (0,\e_0]$ such that, for each $\la\in(0,1)$ and $\e\in(0,\e_1],$ any $T$-periodic solution of the differential equation 
\begin{equation}\label{eq:H}
x'=\e\la F(t,x,\e),\, x\in\ov V,
\end{equation}
 is entirely contained in $V.$
\end{itemize}

\begin{remark}\label{rmk:hyp-h}
In applications, hypothesis ${\bf H}$ can be checked by obtaining a contradiction when its negation is assumed. The negation of hypothesis ${\bf H}$ provides numerical convergent sequences $(\e_m)_{m\in\mathbb{N}}\subset(0,\e_0)$ and $(\la_m)_{m\in\mathbb{N}}\subset(0,1),$ such that $\e_m\to 0$ as $m\to \infty,$ and a sequence of $T$-periodic solutions $x_m(t)\in\ov V$ of $x'=\e_m\la_m F(t,x,\e_m)$ for which there exists $t_m\in[0,T]$ such that $x_m(t_m)\in\partial V$ for each $m\in\mathbb{N}.$ In particular,
\[
x_m(t)=x_m(0)+\e_m\la_m\int_0^t F(s,x_m(s),\e_m) \d s\, \text{ and }\, \int_0^T F(t,x_m(t),\e_m) \d t=0,
\]
for each $m\in\mathbb{N}.$ Furthermore, as an application of  Arzel\'{a}-Ascoli's Theorem, the sequence of functions $(x_m)_{m\in\mathbb{N}}$ can be considered uniformly convergent to a constant function in $\partial V.$

It is worth mentioning that, when the boundary of $V$, $\p V,$ is a smooth manifold, hypothesis {\bf H} holds provided that: ``there exists $\e_1\in(0,\e_0]$ such that, for each $z\in\p V,$ $F(t,z,\e)$ is transversal to $\p V$ at $z,$ for every $t\in[0,T]$ and $\e\in(0,\e_1]$''. Indeed, assume that,  for some $\la\in(0,1)$ and $\e\in(0,\e_1],$ $\f(t)$ is a $T$-periodic solution of $x'=\e\la F(t,x,\e)$ in $\ov V$ which is not entirely contained in $V,$ that is, there exists $\hat t\in[0,T]$ such that $\hat z=\f(\hat t)\in\p V.$ Since $\f(t)\in \ov V$ for every $t\in[0,T],$ we get that $\f'(\hat t)\in T_{\hat z}\p V$ $($tangent space of $\p V$ at $\hat z),$ consequently, $F(\hat t, \hat z,\e)$ is tangent to $\p V$ at $\hat z$. This last sufficient condition, although much more restrictive than hypothesis ${\bf H}$, is more computable and easier to be checked, so it is important to keep it in mind.
\end{remark}

Define the {\it full averaged function}
$f:D\times [0,\e_0]\rightarrow\R^n$ as the average of the right-hand side of \eqref{eq:e1}, that is,
	\[\label{eq:avgd-functions}
	f(z,\e) = \dfrac{1}{T} \int_{0}^{T} \e F(s, z,\e)\d s.
	\]
Our first main result relates the existence of periodic solutions of the differential equation \eqref{eq:e1} to the Brouwer degree of the full averaged function.

\begin{mtheorem}\label{thm:main1} 
Consider the continuous $T$-periodic non-autonomous differential equation \eqref{eq:e1}. Assume that for a given open bounded subset $V\subset \R^n,$ with $\ov V\subset D,$ hypothesis ${\bf H}$ holds,
	\begin{equation}\label{Ha}
	 f(z,\e)\neq 0,\quad \text{ for all }\quad z\in\partial V\quad \text{ and }\quad \e\in(0,\e_1],
	\end{equation}
and $d_B(f(\cdot,\e^*),V,0)\neq 0,$ for some $\e^*\in (0,\e_1].$
	Then, for each $\e\in (0,\e_1],$ there exists a $T-$periodic solution $\f(t,\e)$ of the differential equation \eqref{eq:e1} satisfying $\f(t,\e)\in \ov V,$ for every $t\in[0,T].$ 
\end{mtheorem}

In many situations, derivatives of $F$ with respect to $\e$ up to some order are known. 
In these cases, the differential equation \eqref{eq:e1} writes
	\begin{equation}\label{eq:order-k-ode}
	x^\prime = \sum_{i=1}^k\e^i F_i(t,x)  + \e^{k+1} R(t, x, \e),
	\end{equation}
	where $F_i\colon\R\times  D \to \R^n,$ for $i\in\{1, \ldots, k\},$ and $R\colon\R\times D \times [0,\e_0]\to\R^n $ are continuous functions $T-$periodic in the variable $t.$ Accordingly, define $f_0=0$ and, for each $i\in\{1, \ldots, k\},$ denote by $f_i\colon D \to \R^n$ the average of $F_i,$ that is,
		\[
	f_i(z) = \dfrac{1}{T} \int_{0}^{T} F_i(s, z)\d s.
		\]
	Also, define the {\it $k$-truncated averaged function} $\CF_k \colon  D \times [0,\e_0]\to \R^n$ and the \emph{averaged remainder} $r\colon  D \times [0,\e_0]\to \R^n$, respectively, by
	\[
	\CF_k(z,\e) = \sum\limits_{i = 1}^{k}\e^i f_i(z)\, \text{ and } \, r(z,\e) = \dfrac{1}{T}\int_0^T R(s,z,\e)\d s.
	\]

Our second main result relates the existence of periodic solutions of the differential equation \eqref{eq:order-k-ode} to the Brouwer degree of the $k$-truncated averaged function. This is a continuous (non-Lipschitz) version of the higher order averaging theorem shown in \cite{LliNovTei2014}.

\begin{mtheorem}\label{thm:main2}
Consider the continuous $T$-periodic non-autonomous differential equation \eqref{eq:order-k-ode}. Assume that for a given open bounded subset $V\subset \R^n,$ with $\ov V\subset D,$ hypothesis ${\bf H}$ holds,
	\begin{equation}\label{eq:lemma-hyp}
	\lim_{\e\to 0}\inf_{z\in\p V}\left|\dfrac{\CF_{k}(z,\e)}{\e^{k+1}}\right|>\max\{|r(z,\e)|: (z,\e)\in\ov V\times[0,\e_1]\},
	\end{equation}
and $d_B(\CF_{k}(\cdot,\e),V,0)\neq 0,$ for $\e>0$ sufficiently small. Then, there exists $\ov \e\in(0,\e_1]$ such that, for each $\e\in(0,\ov\e]$, the differential equation \eqref{eq:order-k-ode} has a  $T-$periodic solution $\f(t,\e)$  satisfying $\f(t,\e)\in \ov V,$ for every $t\in[0,T].$
\end{mtheorem}

As a consequence of Theorem \ref{thm:main2} we get our third main result.

\begin{mtheorem}\label{thm:main3}
Consider the continuous $T$-periodic non-autonomous differential equation \eqref{eq:order-k-ode}. Suppose that for some $\ell\in\{1,2,\ldots,k\},$ $f_0=\ldots=f_{\ell-1}=0,$ $f_{\ell}\neq0,$ and let $z^*\in D$ be an isolated zero of $f_{\ell}.$ Assume that there exists a bounded neighbourhood  $V\subset \R^n$ of $z^*,$ with $\ov V\subset D$ and $f_{\ell}(z)\neq0$ for every $z\in \ov{V}\setminus\{z^*\},$ such that hypothesis ${\bf H}$ holds and $d_B(f_{\ell},V,0)\neq 0.$ Then, there exists $\ov \e\in(0,\e_1]$ such that, for each $\e\in(0,\ov\e]$, the differential equation \eqref{eq:order-k-ode} has a $T-$periodic solution $\f(t,\e)$ satisfying $\f(t,\e)\in \ov V,$ for every $t\in[0,T],$ and  $\f(\cdot,\e)\to z^*$ uniformly as $\e\to 0.$
\end{mtheorem}

It is worth mentioning that, for $\ell=1,$ the existence of an isolated zero of $f_1$ ensures that hypothesis ${\bf H}$ holds. Indeed, let $V\subset \R^n$ be a bounded neighbourhood of $z^*,$ with $\ov V\subset D,$ such that $f_{1}(z)\neq0$ for every $z\in \ov{V}\setminus\{z^*\}.$ If hypothesis ${\bf H}$ does not hold on $V$, from Remark \ref{rmk:hyp-h}, we get a sequence of $T$-periodic functions $(x_m)_{m\in\N}$ uniformly converging to a constant function in $\p V,$ let us say $z_0,$ such that 
\[
\int_0^T F(t,x_m(t),\e_m) \d t=0, \quad m\in\N.
\]
Taking the limit in the integral above, we get that $f_1(z_0)=0,$ which contradicts the fact that $f_{1}(z)\neq0$ for every $z\in \ov{V}\setminus\{z^*\}.$ Therefore, \cite[Theorem 1.2]{buicua2004averaging} follows as a corollary of Theorem \ref{thm:main3}, namely:

\begin{corollary}\label{cor:c1}
Consider the continuous $T$-periodic non-autonomous differential equation \eqref{eq:order-k-ode}. Let $z^*\in D$ be an isolated zero of $f_{1}$ and assume that there exists a bounded neighbourhood $V\subset \R^n$ of $z^*,$ with $\ov V\subset D$ and $f_{1}(z)\neq0$ for every $z\in \ov{V}\setminus\{z^*\},$ such that $d_B(f_{1},V,0)\neq 0.$ Then, there exists $\ov \e\in(0,\e_1]$ such that, for $\e\in(0,\ov\e]$, the differential equation \eqref{eq:order-k-ode} has a $T-$periodic solution $\f(t,\e)$ satisfying $\f(t,\e)\in \ov V,$ for every $t\in[0,T],$ and  $\f(\cdot,\e)\to z^*$ uniformly as $\e\to 0.$
\end{corollary}

\smallskip

This paper is structured as follows. Section \ref{sec:prel} contains some basic notions and definitions on degree theory as well as some preliminary results. More specifically, in Section \ref{sec:BD}, we introduce the Brouwer degree for studying zeros of functions defined on finite dimensional spaces; in Section \ref{sec:LSD}, we introduce the Leray-Schauder degree, which is an extension of the Brouwer degree for functions defined on infinite dimensional spaces; in Section \ref{sec:CD}, we introduce the coincidence degree for studying fixed points of operator equations; and in Section \ref{sec:CT}, we discuss a continuation result based on degree theory for solutions of operator equations. Section \ref{sec:proof} is completely devoted to the proof of our main results. Finally, in Section \ref{sec:applic}, we analyze the following continuous higher order perturbation of a harmonic oscillator
\begin{equation}\label{ap1}
\ddot x=-x +\e \big(x^2+\dot x^2)+\e^k \, \dot x  \sqrt[3]{x^2+\dot x^2-1} +\e^{k+1}E(x,\dot x,\e), \vspace{0.2cm}\\
\end{equation}
where $k$ is a positive integer and $E$ is a continuous function on $\R^3.$ Clearly, the differential equation \eqref{ap1} is not Lipschitz in any neighborhood of $\s^1=\{(x,\dot x)\in\R^2:\,x^2+\dot x^2=1\}.$ As an application of our main results, we get  the existence of a periodic solution $x_{\e}(t)$ of \eqref{ap1} satisfying $(x_{\e}(t),\dot x_{\e}(t))\to\s^1$  uniformly  as $\e\to 0$ (see Proposition \ref{prop:ap}). Notice that no previous version of the averaging method could be applied to detect such a periodic solution.

\section{Degree theory and preliminary results}\label{sec:prel} 

This section is devoted to the basic notions and definitions of degree theory as well as some preliminary results.

\subsection{Brouwer degree}\label{sec:BD}
The Brouwer degree is defined as an integer-valued function that assigns to each triple $ (f,V,y_0),$ where $ V\subset\R^n $ is an open bounded subset of $ \R^n, $ $ f:\ov{V}\to \R^n $ is a continuous function, and $y_0\notin f(\p V),$ the number $ d_B(f,V,y_0)$ whose defining properties are:
\begin{enumerate}[label={\bf B.\arabic*},ref={\bf B.\arabic*}]
	\item\label{brouwer:existence} (Existence) If $ d_B(f,V,y_0)\neq 0, $ then $ y_0\in f(V). $ Furthermore, if $ \mathbb{1}:\ov{V}\to\R^n $ is the identity function and $ y_0\in V, $ then $ d_B(\mathbb{1}, V, y_0) = 1. $
	\item\label{brouwer:additivity} (Additivity) If $V_1, V_2\subset V$ are disjoint open subsets of $V$ such 
	that $y_0\notin f(\ov{V}\backslash (V_1\cup V_2)),$ then $$d_B(f, V, y_0) = 
	d_B(f_{\vert {V_1}}, V_1, y_0) + d_B(f_{\vert {V_2}}, V_2, y_0).$$
	\item\label{brouwer:homotopy} (Invariance under homotopy) If $ \{f_t:\ov{V}\to\R^n\, |\, t\in [0,1] \} $ is a continuous homotopy and $ \{ y_t\, |\, t\in [0,1] \} $ is a continuous curve such that $ y_t\notin f_t(\p V),\,\forall t\in [0, 1] $ then $ d_B(f_t,V,y_t) $ is constant in $ t. $
\end{enumerate}

An important property of the Brouwer degree, that follows directly from Property \ref{brouwer:homotopy}, is that it is locally constant, see \cite[Theorem 3.1 (d5)]{deimling1985}:
\begin{itemize}
\item[{\bf B.4}]  (Local constancy) $d_B(g,V,y_0)=d_B(f,V,y_0)$ for every continuous function $g:\ov{V}\to \R^n $ such that $ |g-f| < \dist(y_0,f(\partial V)). $
\end{itemize}

Another result concerning the invariance of the Brouwer degree under small perturbations, that we shall use later on, is the following:

\begin{lemma}[{\cite[Lemma 4]{candido-libre-novaes-2017}}]\label{NNL}
Let $V$ be an open bounded subset of $\R^m.$ Consider the continuous
functions $f_i:\ov{V}\to \mathbb{R}^n,$ $i\in\{0,1,\cdots,\kappa\},$ and
$f,g,r:\ov{V}\times [0,\e_0] \rightarrow \mathbb{R}^n$ given by
\[
g(z,\e)=f_0(z)+\e f_1(z)+\cdots+\e^\kappa f_\kappa(z) \mbox{ and }
f(z,\e)=g(z,\e)+\e^{\kappa+1}r(z,\e).
\]
Let $V_\e \subset V,$ $R=\max\{\vert r(z,\e) \vert :
(z,\e)\in\ov{V}\times [0,\e_0]\}$ and assume that $\vert g(z,\e)
\vert> R\vert \e\vert^{\kappa+1}$ for all $ z \in \p V_\e $ and
$\e\in(0,\e_0].$ Then, for each
$\e\in(0,\e_0]$ we have $d_B\left(f(\cdot,
\e),V_\e,0 \right)=d_B\left(g(\cdot, \e),V_\e,0  \right).$
\end{lemma}

\subsection{Leray-Schauder degree}\label{sec:LSD}
The Leray-Schauder degree was introduced by Leray and Schauder \cite{LerayShcauder1934} in the context of compact perturbations of the identity on normed linear spaces. 

\begin{definition}[{\cite{cronin1964fixed},\cite[Theorem 8.1]{deimling1985}}]\label{def:lsd}
Let $ X $ be a real normed linear space, $ \O\subset X $ be an open bounded subset of $X,$ and $ M:\ov{\O}\to X $ be a compact mapping. If $ y_0\notin (\Id - M)(\p\O), $ then the Leray-Schauder degree is an integer-valued function defined by
\begin{align*}
d_{LS}(\Id - M, \O, y_0) = d_B((\Id - M_1)_{\vert \O\cap X_1}, \O\cap X_1, y_0),
\end{align*}
where $ M_1:\ov{\O}\to X $ is any compact mapping satisfying  \[\sup_{x\in\ov{\O}}|M_1 x - Mx| < \dist (y_0,(\Id - M)(\p\O)),\]  and $ X_1 $ is any finite dimensional subspace of $X$ such that $ y_0\in X_1$ and $M_1(\ov\Omega)\subset X_1$.
\end{definition}

Considering the setting in Definition \ref{def:lsd}, one can prove that the Leray-Schauder degree has the following properties (see, for instance \cite{deimling1985}):

\begin{enumerate}[label={\bf LS.\arabic*},ref={\bf LS.\arabic*}]
	\item\label{leray-schauder:existence} (Existence) If $ d_{LS}(\Id - M, \O, y_0)\neq 0, $ then there exists $ x\in\O $ such that $ x - Mx = y_0. $
	\item\label{leray-schauder:additivity} (Additivity) If $ \O_1, \O_2\subset \O $ are open disjoint subsets of $ \O $ such that $ y_0\notin (\Id - M)(\ov{\O}\backslash (\O_1\cup\O_2), $ then 
	\begin{multline*}
		d_{LS}(\Id - M, \O, y_0) = d_{LS}((\Id - M)_{\vert \O_1}, \O_1, y_0) + d_{LS}((\Id - M)_{\vert \O_2}, \O_2, y_0).
	\end{multline*}
	\item\label{leray-schauder:homotopy} (Invariance under homotopy) Let $ H:\ov{\O}\times[0,1]\to X $ be a compact mapping and $ \{ y_t\in X\, |\, t\in [0,1] \} $ be a continuous curve such that $ x - H(x,t)\neq y_t, $ for all $ (x,t)\in\p\O\times [0,1]. $  Then, $d_{LS}(\Id- H(\cdot, t), \O, y_t)$ is constant in $ t. $
\end{enumerate}

One can readily see the similarity between these properties and the properties \ref{brouwer:existence}-\ref{brouwer:homotopy}. Indeed, as we can see from Definition \ref{def:lsd}, the Leray-Schauder degree is obtained from the Brouwer degree by approximating the infinite dimensional space $ X $ by finite dimensional ones. In particular, in this scenario, if $ X $ is finite dimensional, then $d_{LS}(I - M,\O,y_0) = d_B(I - M, \O, y_0).$ 

\subsection{Coincidence degree}\label{sec:CD}
Finally, consider two real normed vector spaces $ X $ and $ Z,$ and  $\dom L$ a subspace of $X.$ Let $ L:\dom L\subset X\to Z $ be a  linear mapping, $\O$ an open bounded subset of $ X,$ and $ N:\ov{\O}\subset X\to Z $ any mapping. The coincidence degree concerns the existence of solutions of the operator equation  
\[
Lx = Nx, \,\, x\in\ov \Omega,
\] under suitable assumptions on $L,$ $N,$ and $\O.$

We say that $ L $ is a {\it Fredholm mapping} of index $0$ if $ \Im L $ is a closed subset of $ Z ,$ $ \ker L$ and $\coker L = Z/\Im L $ are finite dimensional, and $ \dim\ker L = \dim\coker L. $ 

Now, we introduce the concept of $ L- $compact mapping for a  Fredholm mapping $L$ of index $0.$ 
Let  $ P:X\to X $ and $ Q:Z\to Z $ be continuous projections such that the sequence 
\begin{equation}\label{eq:exact-seq}
\begin{tikzcd}
X \arrow{r}{P} & \dom L \arrow{r}{L} & Z \arrow{r}{Q} & Z
\end{tikzcd}
\end{equation}
is exact, i.e. $\Im P = \ker L$ and $\Im L = \ker Q.$ It can be seen that $ L_P = L_{\vert \ker P~\cap~\dom L} $ is an isomorphism. Therefore, we can take its inverse, denoted by $ K_P, $ and define the generalized inverse of $ L $ by $K_{P,Q} = K_P (\Id- Q). $ Also, denote by $ \Pi:Z\to \coker L $ the canonical projection that sends any $ y\in Z $ onto its equivalence class in $ \coker L. $ Accordingly, we say that a mapping  $N:\ov{\O}\subset X\to Z $ is {\it $ L- $compact} on $\ov{\O}$ if the mappings $\Pi N:\ov{\O}\subset X\to \coker L$ and $ K_{P,Q} N:\ov{\O}\subset X\to X$ are compact on $\ov{\O}$, that is, $\Pi N$ and $ K_{P,Q} N $ are continuous on $\ov{\O}$ such that $\Pi N(\ov{\O}) $ and $ K_{P,Q} N (\ov{\O}) $ are relatively compact. At this point it is worth noting that the projectors $ P $ and $ Q $ are not unique in general, but one can prove that the definition of $ L- $compactness does not depend upon the choices of $ P $ and $ Q. $

The next proposition is a key result for the definition of coincidence degree.

\begin{proposition}[{\cite[Proposition III.0]{gaines1977coincidence}}]\label{Prop:fixedpoint}
Let $ L:\dom L\subset X\to Z $ be a linear mapping. If there exists a linear injective mapping $ \L:\coker L\to\ker L, $ then $ Lx = y $, for some $y\in Z,$ if, and only if, 
$ (\Id - P)x = (\Lambda\Pi + K_{P, Q})y.$
\end{proposition}

For $y=Nx,$ this proposition says that, as long as there exists a linear injective mapping $ \L:\coker L\to\ker L, $
  the set of solutions of $ Lx = Nx $ is equal to the set of fixed points of the mapping 
  \[
  M = P + (\Lambda\Pi + K_{P, Q})N.
  \] Moreover, $M$ can be proven to be a compact mapping  provided that $L$ is a Fredholm mapping of index $0$ and $N$ is $ L- $compact on $\ov\Omega$ (see \cite[Propositions III.2 and III.3]{gaines1977coincidence}). Accordingly, if one has, in addition, that $ 0 \notin (L - N)(\p \O\cap \dom L), $ then the Leray-Schauder degree of $ \Id - M $ with respect to $ \O $ and 0, $ d_{LS}(\Id - M,\O,0), $ is well-defined. This motivates the following definition:

\begin{definition}\label{def:cd}
	Let $ L:\dom L\subset X\to Z $ be a linear Fredholm mapping of index $0$ and $ N:\ov{\O}\subset X\to Z $ be an $L$-compact mapping on $\ov \Omega$.  The coincidence degree of $ L \mbox{ and } N $ with respect to $ \O $ is defined as $ d((L,N),\O) := d_{LS}(\Id - M,\O,0). $
\end{definition}

\begin{remark}
 In the definition above, it is worthwhile to point out that  $ |d((L,N),\O)| $ only depends  on $ L, N $ and $ \O. $ The sign of $ d((L,N),\O) $ depends on wether or not $ \Lambda $ is an orientation preserving isomorphism.
\end{remark}

	In the above setting, the properties \ref{leray-schauder:existence}-\ref{leray-schauder:homotopy} of the Leray-Schauder degree induce the following properties on the coincidence degree.
\begin{enumerate}[label={\bf C.\arabic*},ref={\bf C.\arabic*}]
	\item\label{coincidence:existence} If $ d((L,N),\O)\neq 0, $ then there exists $ x\in\O $ such that $ Lx = Nx. $
	\item\label{coincidence:additivity} If $ \O_1,\O_2\subset\O $ are open disjoint subsets of $ \O, $ then 
	\[d((L,N),\O) = d((L,N),\O_1) + d((L,N),\O_2). \]
	\item\label{coincidence:homotopy} Let $ N:\ov{\O}\times [0, 1]\to Z $ be a $ L- $compact mapping on $ \ov{\O}\times [0,1] $ such that $0 \notin (L - N(\cdot,t))(\dom L\cap \p \O)$ for each $ t \in [0, 1].$ Then, $ d((L,N(\cdot,t)),\O) $ is constant in $ t. $
\end{enumerate}

\subsection{A continuation theorem}\label{sec:CT}

Let $ L:\dom L\subset X\to Z $ be a linear Fredholm mapping of index $0$ and consider  $ P:X\to X $ and $ Q:Z\to Z $  continuous projections such that the sequence \eqref{eq:exact-seq} is exact, that is,  $ \Im P = \ker L $  and $ \Im L = \ker Q.$ Let $ {N}:\ov{\O}\times [0,1]\to Z $ be an $ L- $compact mapping on $ \ov{\O}\times [0,1]$, where $ \O\subset X $ is open and bounded. 

In this section, following  \cite[Chapter IV]{gaines1977coincidence} closely, we discuss sufficient conditions in order to guarantee that the operator equation
\begin{align}\label{eq:cor-iv1:operator-eq}
Lx = \la {N}(x,\la),\,\, (x,\la)\in\ov\Omega\times[0,1],
\end{align}
has solutions for each $\la\in[0,1].$ The following proposition is an important tool in this quest.

	\begin{lemma}[{\cite[Lemma IV.1]{gaines1977coincidence}}]\label{lemma:basic-lemma-02}
For each $\la\in (0,1],$ the set of solutions of \eqref{eq:cor-iv1:operator-eq} coincides with the set of solutions of
\[
Lx = \wt N(x,\la):= QN(x,\la) + \la (\Id - Q)N(x,\la).
\]
For $ \la=0, $ every solution of the latter equation is a solution of \eqref{eq:cor-iv1:operator-eq}.
\end{lemma}

Accordingly, based on the coincidence degree theory discussed in the previous section, we shall compute the coincidence degree $d((L,\wt N(\cdot,\la)),\Omega),$ for each $\la\in[0,1].$
Thus, assuming
\begin{itemize}
\item[{\bf A.1}]\label{itm:cor-iv1-1} $ Lx\neq \la {N}(x,\la), $ for every $ x\in\dom L\cap\partial\O$ and $ \la\in (0,1);$ and

\item[{\bf A.2}]\label{itm:cor-iv1-2} $ Q{N}(x,0)\neq 0, $ for every $ x\in\ker L\cap\partial\O,$
\end{itemize}
we have that either $Lx = \wt N(x,1),$ for some $x\in\dom L\cap\p\Omega,$ or
\begin{equation}\label{claim}
d((L,\wt N(\cdot,\la)),\Omega)=d_B(JQ{N}(\cdot,0)_{|_{\ker L\cap\O}},\ker L\cap\O,0),
\end{equation}
for each $\la\in[0,1],$ where  $J:\Im Q\rightarrow \ker L$ is an isomorphism.

 Indeed, it is straightforward to see that $\wt N$ is $L$-compact on $\ov\Omega\times[0,1]$. Assuming {\bf A.1}, {\bf A.2}, and $Lx \neq \wt N(x,1),$ for every $x\in\dom L\cap\p\Omega,$ and taking Lemma \ref{lemma:basic-lemma-02} into account, one can see that $0 \notin (L - \widetilde N(\cdot,\la))(\dom L\cap \p \O),$ for each $\la \in [0, 1].$ Thus, by property {\bf C.3}, we get
		\begin{equation}\label{lemma:eq1}
			d((L,\wt N(\cdot,\lambda),\O)= d((L,QN(\cdot,0)),\O),
		\end{equation}
		for each $\la \in [0, 1].$
		By definition of the Coincidence Degree (see Definition \ref{def:cd}), we have
\begin{equation}\label{lemma:eq2}
\begin{aligned}
			d((L,QN(\cdot,0)),\O) & = d_{LS}(\Id - P - (\L\Pi + K_{P,Q})QN(\cdot,0),\O,0)\\
			& = d_{LS}(\Id - P - \L\Pi QN(\cdot,0),\O,0),
			\end{aligned}
\end{equation}
		where, we recall, $ \Pi:Z\to \coker L $ is the canonical projection and 
 $ \L:\coker L\to\ker L $ is an isomorphism. Therefore, applying the definition of the Leray-Schauder Degree  (see Definition \ref{def:lsd})  for $ X_1 = \ker L $ and $ M_1 = M = P + \L\Pi QN(\cdot,0),$ and using the fact that $ (\Id - P)_{\vert\ker L} = 0, $ we obtain that
		\begin{equation}\label{lemma:eq3}
			\begin{aligned}
			 d_{LS}(\Id - P -& \L\Pi  QN(\cdot,0),\O,0)\\
			    &= d_B((\Id - P - \L\Pi Q N(\cdot,0))_{\vert \O\cap\ker L},\O\cap\ker L,0)\\
				& = d_B(-\L\Pi Q {N(\cdot,0)}_{\vert \O\cap\ker L},\O\cap\ker L,0),\\
				&= d_B(JQ {N(\cdot,0)}_{\vert\O\cap\ker L},\O\cap\ker L, 0),
			\end{aligned}
		\end{equation}
		where $J=-\L\Pi_Q$ and $\Pi_Q := \Pi_{\vert \Im Q} $ are isomorphisms. 
		Taking the relationships \eqref{lemma:eq1},  \eqref{lemma:eq2},  and \eqref{lemma:eq3} into account, we get \eqref{claim}.

Notice that, in the reasoning above, we are fixing the isomorphism $ \L:\coker L\to\ker L $ and choosing  $ J = -\L\Pi_Q.$  Nevertheless, since $ \L$ is arbitrary and $\Pi_Q$ is an isomorphism, we could fix any isomorphism $J:\Im Q\rightarrow \ker L$ and choose $\L=-J\Pi_Q^{-1}.$ 

Then, we get the following continuation result, which was proven in \cite{gaines1977coincidence}:
\begin{theorem}[{\cite[Corollary IV.1]{gaines1977coincidence}}]\label{cor:iv.1}
	In addition to condition {\bf A.1} and {\bf A.2}, assume that 
$ d_B(JQ{N}(\cdot,0)_{|_{\O\cap\ker L}},\O\cap\ker L,0)\neq 0. $
	Then, the operator equation \eqref{eq:cor-iv1:operator-eq} admits a solution, which lies in $\O$ $($resp. $\ov \O)$ for $\la\in[0,1)$ $($resp. $\la=1).$ 
\end{theorem}

It is worth mentioning that, in the construction of the Brouwer degree performed in Section \ref{sec:BD}, we are tacitly assuming that the involved spaces are not $ 0- $dimensional. However, the Brouwer degree can be extended to the $0-$dimensional scenario by defining 
$d_B(\Id,\{0\},0)=1$ and $d_B(\Id,\emptyset,0)=0$ (see \cite[Section IV]{gaines1977coincidence}). With that in mind, Theorem \ref{cor:iv.1} also holds when $\ker L=\{0\}$. Indeed, in this case, $P=0,$ $Q=0,$ $\Pi=0,$ and $K_{P,Q}=L^{-1}.$ Thus, on can see that conditions {\bf A.2} and $ d_B(JQ{N}(\cdot,0)_{|_{\O\cap\ker L}},\O\cap\ker L,0)\neq 0$ are equivalent to $0\notin \p\Omega$ and $0\in \O$, respectively. Therefore, going back to relationship \eqref{lemma:eq2}, we obtain 	$d((L,QN(\cdot,0)),\O) = d_{LS}(\Id,\O,0)=1$, and then Theorem \ref{cor:iv.1} follows.

\section{Proof of the main results}\label{sec:proof}

We denote by $C[0, T]$ the space of all continuous functions defined in $[0, 
T]$ with values in $ \R^n $ and define the function spaces \[ C_0 = \{ x\in C[0, T]~\colon~ x(0) = 0\}\quad\mbox{ and }\quad C_T = \{ x\in C[0, T]~\colon~ x(0) = 
x(T)\},\] both endowed with the sup-norm making them into real Banach spaces. Set $X=C_T$ and $Z=C_0$ and, for a given open bounded subset $V$ of $\R^n,$ take $ \O = \{ x\in C_T\colon x(t)\in V,\,\forall\, t\in [0, T] \}, $ which is an open bounded subset of $ C_T. $

Define the linear mapping $ L\colon C_T\to C_0 $ by
\[
Lx(t) = x(t)- x(0),
\]
 and, for each $\e\in (0,\e_0],$  define $ N_{\e}:\ov{\O}\to C_0 $ by $$ N_{\e}(x)(t) = \int_0^t  \e \, F(s,x(s),\e)\d s. $$

Notice that a function $x\in C_T$ is a $ T-$periodic solution of the differential equation \eqref{eq:e1} in $\ov V$ if, and only if, it is a solution of the operator equation
\begin{equation}\label{eq:functional-eq}
Lx =N_{\e}(x), \,\, x\in\ov\Omega.
\end{equation}

\begin{proof}[{\bf Proof of Theorem \ref{thm:main1}}]

Consider $N_{\e}(x,\la)=N_{\e}(x).$ In order to obtain the existence of a solution of the operator equation \eqref{eq:functional-eq} and conclude this proof, we shall apply Theorem \ref{cor:iv.1}  for $\lambda=1$ to the operator equation
\begin{equation}\label{eq:Nlambda}
Lx =\lambda N_{\e}(x,\la),\,\, (x,\la)\in\ov\Omega\times[0,1].
\end{equation}

Firstly, we must check that $L$ is a Fredholm mapping of index $0$ and that $ N_{\e} $ is $ L- $compact on $\ov\Omega\times[0,1]$, for each $ \e\in (0,\e_0].$
Notice that  $\Im L=C_T\cap C_0,$ which is closed in $C_0.$ In addition,
\[
 \ker L = \big\{ x\in C_T:x(t) = z,\, z\in \R^n\big\},
\] 
 that is, the space of all constant function in $ \R^n,$ which can be identified with $ \R^n,$ and
\[
\coker L =\big \{[y]:=y+\Im L:y\in C_0\big\}.
\]
One can readily see that 
 $[y_1] = [y_2] $ if, and only if, 
 $ y_1(T) = y_2(T),$ which means that $ \coker L $ can also be identified with $ \R^n. $ 
Hence,  $\dim \ker L=\dim \coker L$ and, therefore, $L$ is a Fredholm mapping of index $0.$ Moreover, the natural projection $ \Pi:C_0\to\coker L$ is given by  $ \Pi y = [y(T)].$
Now, in the above setting, consider the continuous projections
  $ P\colon C_T\to C_T$ and $Q\colon C_0\to C_0 $ defined by
 \[ 
 Px(t) = x(0)\,\, \text{ and }\,\, Qy(t) = \dfrac{t~y(T)}{T}, \,\, \text{ for }\,\, t\in [0,T], 
 \]
 respectively, and let $\L:\coker L\to\ker L $ be defined by $\L[z](t) = -z,$ for $t\in [0,T].$
 From here, it is straightforward to check that $ \Im P = \ker L $  and $ \Im L = \ker Q, $ which implies that the sequence in \eqref{eq:exact-seq} is exact, and that $ N_{\e} $ is $ L- $compact on $\ov\Omega\times[0,1],$ for each $ \e\in (0,\e_1].$ 

In addition, $ \Im Q=\big\{ x\in C_0:x(t) = t\, v,\, v\in\R^n \big\},$ 
which can be identified with $ \R^n.$ Thus, consider the isomorphism $J:  \Im Q\rightarrow \ker L$ given by 
\[
J y(t)=\dfrac{y(T)}{T}.
\] 

Now, we are in position of checking the conditions to apply Theorem \ref{cor:iv.1} .

Notice that $Lx=\la N_{\e}(x,\la)$ for some $x\in\ov\Omega,$ $\la\in[0,1],$ and $\e\in(0,\e_1]$
if, and only if, $x$ is a $ T-$periodic solution of the differential equation \eqref{eq:H}. Thus, taking hypothesis ${\bf H}$ into account, we get that, for each $\e\in(0,\e_1],$ $Lx\neq\la N_{\e}(x,\la),$ for every $ x\in\dom L\cap\partial\O $ and $ \la\in (0,1)$. Therefore,  condition {\bf A.1} of Theorem \ref{cor:iv.1}  holds, for each $\e\in(0,\e_1].$

 		In addition, for $ x\in\ker L\cap\partial\O, $ that is, $ x(t)\equiv z\in\partial V,$ 
				\begin{align*}
		QN_{\e}(x,0)(t) = \dfrac{t}{T}\int_0^T \e F(s,z,\e)\d s =t f(z,\e),
		\end{align*}
		which, by hypothesis, is not the $0$ constant function in $C_0.$ Therefore,  condition {\bf A.2}  of Theorem \ref{cor:iv.1}  holds, for each $\e\in(0,\e_1].$

		 Finally, for $ x\in\ker L\cap\O, $ say $ x(t)\equiv z\in V, $ we have
		$JQN_{\e}(x,0) = f(z,\e). $
		Thus, \[ d_B(JQN_{\e}(\cdot,0)_{|\ker L\cap\O},\ker L\cap\O,0) = d_B(f(\cdot,\e),V,0). \]
		We claim that $d_B(f(\cdot,\e),V,0)\neq0$ for each $\e\in(0,\e_1].$ Indeed, denote $
		\mathcal{E}=\{\e\in(0,\e_1]:\,d_B(f(\cdot,\e),V,0)\neq0\}.
		$
		By hypothesis, there exists $\e^*\in(0,\e_1]$ such that $d_B(f(\cdot,\e^ *),V,0)\neq0,$ thus $\mathcal{E}\neq\emptyset.$ Moreover, given $\hat\e\in\mathcal{E}$, by hypothesis \eqref{Ha} and compactness of $\p V$, there exists a small open interval $\mathcal{I}$ containing $\hat\e$ such that $f(z,\e)\neq0$ for all  $z\in\partial V$ and $\e\in \mathcal{I}$ and, then, from Property {\bf B.4}, $\mathcal{I}$ can be taken smaller if necessary in order that $d_B(f(\cdot,\e),V,0)=d_B(f(\cdot,\hat\e),V,0)\neq 0,$ for every $\e\in \mathcal{I}.$ Thus, $\mathcal{I}\cap(0,\e_1]\subset  \mathcal{E}$, which means that $\mathcal{E}$ is open in $(0,\e_1].$ Analogously, one can see that $(0,\e_1]\setminus\mathcal{E}=\{\e\in(0,\e_1]:\,d_B(f(\cdot,\e),V,0)=0\}$ is open in $(0,\e_1]$ and, consequently, $\mathcal{E}$ is also closed in $(0,\e_1].$ Hence, from the connectedness of $(0,\e_1]$, we obtain $\mathcal{E}=(0,\e_1].$ 
		
		Therefore, we conclude that all conditions of Theorem \ref{cor:iv.1}  hold, for each $\e\in(0,\e_1].$ Hence, applying Theorem \ref{cor:iv.1}  for $\lambda=1$ to the operator equation \eqref{eq:Nlambda} for each $\e\in(0,\e_1]$ , we get the existence of a solution of the operator equation \eqref{eq:functional-eq} and, consequently, a  $T-$periodic solution $\f(t,\e)$ of the differential equation \eqref{eq:e1}, for each $\e\in(0,\e_1],$   such that $\f(t,\e)\in \ov V$ for every $t\in[0,T].$ 		
\end{proof}

\begin{remark}
In the proof of Theorem A, the maps $L$ and $N_{\e}$ are not the unique possibility for obtaining the result. Indeed, as pointed out by an anonymous referee, one could take $\dom L=\{x\in C_T:\, x \text{ is differentiable}\},$   $Lx(t)=\dot x(t),$ and $N_{\e}(x)(t)=\e F(t,x(t),\e).$ Then, by a suitable choice of the projectors $P$ and $Q$, the proof would follow analogously.
\end{remark}

 \begin{proof}[{\bf Proof of Theorem \ref{thm:main2}}]
For system \eqref{eq:order-k-ode}, we have
\[ f(z,\e) = \dfrac{1}{T}\int_{0}^{T}\left( \sum_{j=1}^{k} \e^j F_j(s,z) + \e^{k+1}R(s,z,\e)\right)\d s = \CF_k(z,\e) + \e^{k+1} r(z,\e). \]
By hypothesis \eqref{eq:lemma-hyp}, there exists $\ov \e\in(0,\e_1]$ such that
$$\left\lvert \CF_k(z,\e)\right\rvert > \lvert\e^{k+1}\rvert\, \max\{ |r(z, \e)|\colon (z,\e)\in \ov{V}\times [0,\e_1]\},$$ for every $ \e \in (0,\ov \e].$
In particular, hypothesis \eqref{Ha} of Theorem \ref{thm:main1} holds for $\e\in(0,\ov\e].$ 
In addition, taking $V_{\e}=V$ in Lemma \ref{NNL}, we get that $d_B(f(\cdot, \e),V, 0) = d_B(\CF_k(\cdot, \e), V, 0)\neq0,$  for $\e\in(0,\ov \e].$
Finally, by hypothesis, $d_B(\CF_k(\cdot, \e), V, 0)\neq 0$ for $\e>0$ sufficiently small. Thus applying a topological argument, analogous to the one used at the end of the proof of Theorem \ref{thm:main1}, we get \[d_B(f(\cdot, \e),V, 0)=d_B(\CF_k(\cdot, \e), V, 0)\neq 0,\] for every $\e\in (0,\ov \e].$  From here, the result follows from Theorem \ref{thm:main1}.
\end{proof}

\begin{remark}\label{estimationTB}
In the proof of Theorem \ref{thm:main2}, one can see that $\ov\e$ can be chosen to be any value in $(0,\e_1]$ such that
\[
\inf_{z\in\p V}\left\lvert \CF_k(z,\e)\right\rvert > \lvert\e^{k+1}\rvert\, \max\{ |r(z, \e)|\colon (z,\e)\in \ov{V}\times [0,\e_1]\},
\]
for every $\e\in(0,\ov\e]$. This provides a way for estimating the interval of the parameter $\e$ where we have ensured the existence of a $T$-periodic solution of the differential equation \eqref{eq:order-k-ode}.
\end{remark}

\begin{proof}[{\bf Proof of Theorem \ref{thm:main3}}]

Without loss of generality, we can assume that $\ell=k.$ Consider neighbourhoods $ V_\mu = B(z^*,\mu)\subset V,$ for $ \mu > 0 $ sufficiently small. Clearly $ V_\mu\to\{ z^*\} $ as $ \mu\to 0. $ Now, $ f_0 = \cdots = f_{k-1} = 0,$ then $\CF_k(\cdot,\e)=\e^k f_k(z).$
 Moreover, since $f_k(z)\neq0,$  for every $z\in\p V_\mu$ and $ r(z,\e) $ is continuous, consequently, bounded on compact sets, we conclude that 
 \[
 \begin{array}{rl}
 \ds\lim_{\e\to 0} \inf_{z\in\p V_{\mu}} \left\lvert \dfrac{\e^k f_k(z)}{\e^{k+1}} \right\rvert =&\ds\!\!\! \lim_{\e\to 0} \inf_{z\in\p V_{\mu}} \left\lvert \dfrac{f_k(z)}{\e} \right\rvert = \infty \\
 >&\ds\!\!\!\max\{|r(z,\e)|: (z,\e)\in\ov V\times[0,\e_1]\}.
 \end{array}
 \]
 Thus, hypothesis \eqref{eq:lemma-hyp} of Theorem \ref{thm:main2} holds. In addition, for every $\e>0,$ we have $ d_B(\CF_k(\cdot,\e),V_\mu,0) = d_B(f_k(z),V_\mu,0),$ which, is distinct from zero, by hypothesis. Hence, by Theorem \ref{thm:main2}, there exists $ \ov \e_\mu > 0 $ and a $ T- $periodic solution $ \f(\cdot,\e) $ of \eqref{eq:order-k-ode} such that $ \f(t,\e)\in \ov V_\mu,\,\forall\, t\in [0, T] $ and for each $ \e\in (0,\ov \e_\mu]. $ Now, given any $ \xi > 0, $ put $ \mu = \xi/2 $ and $ \de = \ov \e_\mu. $ By the conclusion above, $ 0<\e < \de $ implies $ \sup_{t\in [0,T]} |\f(t,\e) - z^*| \leq \xi/2< \xi. $ That is precisely to say that $ \f(\cdot,\e) \to z^*$  uniformly  as $\e\to 0.$ This completes the proof.
\end{proof}

\begin{remark}\label{estimationTC}
In Theorem \ref{thm:main3}, taking Remark \ref{estimationTB} into account, one can see that, for any $\ov\e\in(0,\e_1]$ such that
\[
0<\ov \e<\dfrac{1}{M_{\ell}}\min_{z\in\p V}|f_{\ell}(z)|,
\]
where 
\[
M_{\ell}=\max\{|f_{\ell+1}(z)+\cdots+\e^{k-\ell-1} f_{k}(z)+\e^{k-\ell}r(z,\e)|:\,(z,\e)\in\ov V\times[0,\e_1]\},
\]
a $T-$periodic solution $\f(t,\e)\subset\ov V$ of differential equation \eqref{eq:order-k-ode} exists for every $\e\in(0,\ov\e].$ It is worth mentioning that, since $\ov\e\in(0,\e_1],$ the estimation above is helpful only when the value of $\e_1$ is known, which is established by hypothesis ${\bf H}$.  Remark \ref{rmk:hyp-h} provides a route to prove the existence of such $\e_1,$ however, estimating its value is not always possible. 
\end{remark}

\section{Non-Lipschitz perturbation of a harmonic oscillator}\label{sec:applic}

Consider the continuous higher order perturbation of a harmonic oscillator \eqref{ap1},
\[
\ddot x=-x +\e \big(x^2+\dot x^2)+\e^k \,  \dot x  \sqrt[3]{x^2+\dot x^2-1} +\e^{k+1}E(x,\dot x,\e),
\]
where $k$ is a positive integer and $E$ is a continuous function on $\R^3.$ Clearly, the differential equation \eqref{ap1} is not Lipschitz in any neighborhood of $\s^1=\{(x,\dot x)\in\R^2:\,x^2+\dot x^2=1\}.$ In the next result, Theorem \ref{thm:main3} is applied to show the existence of a periodic solution $x_{\e}(t)$ of \eqref{ap1} satisfying $(x_{\e}(t),\dot x_{\e}(t))\to\s^1$  uniformly  as $\e\to 0.$ Notice that such a periodic solution is not detectable by any Lipschitz version of averaging method.  
\begin{proposition}\label{prop:ap}
For any positive integer $k$ and $|\e|\neq0$ sufficiently small, the differential equation \eqref{ap1} admits a periodic solution $x(t;\e)$ satisfying $(x(t;\e),\dot x(t;\e))\to \s^1$  uniformly  as $\e\to 0.$
\end{proposition}
\begin{proof}
Changing to polar coordinates  $x=r \cos\T, \dot x=-r\sin\T$ and taking $ \theta $ as the new independent variable, the differential equation \eqref{ap1} becomes
\begin{equation}\label{eq:polar}
\begin{array}{rl}
\dfrac{\d r}{\d\T}=\e F(\T,r,\e),
\end{array}
\end{equation}
where
\[
\begin{array}{rl}
F(\T,r,\e)=& -\displaystyle\sum_{i=1}^{k-1}\e^{i-1} r^{i+1}\cos^{i-1}\T\sin\T\vspace{0.2cm}\\
&+\e^{k-1} r\Big(\sqrt[3]{r^2-1}\sin\T -r^k\cos^{k-1}\T\Big)\sin\T+\e^{k}R(\T,r,\e),
\end{array}
\]
which is not Lipschitz in any neighbourhood of $ r = 1.$ Let $V=(1-\al,1+\al)$ for some $0<\al<1.$ Notice that 
\[
f_i=0,\,\text{  for }\, i\in\{1,2,\ldots,k-1\},\,\text{ and }\, f_k(r)=\dfrac{r\sqrt[3]{r^2-1}}{2}.
\]
Moreover, $ f_k $ has a unique positive zero $r^* = 1 $ and is homotopic to the mapping $r\mapsto r-1$ in $ V .$ Therefore, $ d_B(f_k,V,0) \neq 0. $

In what follows we shall assume that $\e>0.$ The result for $\e<0$ can be obtained analogously just by considering $-F(\T,r,\e).$ In order to apply Theorem \ref{thm:main3}, it remains to check hypothesis $ {\bf H}.$ From Remark \ref{rmk:hyp-h}, the negation of hypothesis ${\bf H}$ provides numerical convergent sequences $(\e_m)_{m\in\mathbb{N}}\subset(0,\e_0]$ and $(\la_m)_{m\in\mathbb{N}}\subset(0,1),$ such that $\e_m\to 0$ as $m\to \infty,$ and a sequence $(r_m)_{m\in\N}$ of $2\pi$-periodic solutions of
	\begin{equation}\label{eq:homot-eq}
		r'=\e_m\la_m F(\T,r,\e_m),\quad r\in\ov V,
	\end{equation}
for which there exists $\T_m\in[0,2\pi]$ such that $r_m(\T_m)\in\partial V$ for each $m\in\mathbb{N}.$ As an application of  Arzel\'{a}-Ascoli's Theorem, the sequence of functions $(r_m)_{m\in\mathbb{N}}$ can be taken uniformly convergent to a constant function $r_0\in\partial V=\{1\pm\al\}.$ In particular,
\[\label{contra:edo}
\int_0^{2\pi} F(\T,r_m(\T),\e_m)\d\T=0,
\]
which implies that
\begin{equation}\label{contra:rela}
\int_0^{2\pi} r_m(\T)\sqrt[3]{r_m(\T)^2-1}\sin^2\T \d\T=\dfrac{1}{\e_m^{k-1}}\sum_{i=1}^{k}\e_m^{i-1} G_m^{i+1,i-1} +\CO(\e_m),
\end{equation}
where
\[ G^{i,j}_m = \int_0^{2\pi} r_m(\t)^i\cos^j(\T)\sin(\T)~\d\T. \]
Here, although $ (\e_m)_{m\in\N} $ is a \emph{numerical} sequence, we borrow the Landau's symbol notation $ h_m=\CO(\e_m^p), $ for some $ p\in\N, $ to mean that there exists a positive constant $ C $ such that $ |h_m|\leq C |\e_m^p|, $ for $ m $ sufficiently large.
Note that, by applying integration by parts and using that $ r_m(\T) $ is $ 2\pi- $periodic, we obtain
\begin{align*}
	G_m^{i,j} & = -\dfrac{i}{j+1}\int_0^{2\pi} r_m(\t)^{i-1}\cos^{j+1}(\t)r_m'(\t)~\d\t.
\end{align*}
Since $r_m(\t)$ satisfies \eqref{eq:homot-eq}, we conclude that
\begin{align*}
	G_m^{i,j}& = -\dfrac{i}{j+1}\sum_{l=1}^{k-1}\la_m\e_m^l\int_0^{2\pi}r_m(\t)^{j+l}\cos^{j+l}(\t)\sin\t~\d\t +\CO(\e_m^k)\\
		& = -\dfrac{i}{j+1}\sum_{l=1}^{k-1}\la_m\e_m^l G_{i+l,j+l}(r_m) +\CO(\e_m^k) \\
		& = \CO(\e_m).
\end{align*}
Applying the above procedure recursively, we conclude that 
$G_m^{i,j} =\CO(\e_m^k).$ Thus, from \eqref{contra:rela}, we get that
\[
\int_0^{2\pi} r_m(\T)\sqrt[3]{r_m(\T)^2-1}\sin^2\T \d\T=\CO(\e_m).
\]
Since $r_m\to r_0\in\{1\pm\alpha\}$ uniformly, we compute the limit of the integral above as
\[
\int_0^{2\pi} r_0\sqrt[3]{r_0^2-1}\sin^2\T \d\T=0,
\]
which is an absurd, because
\[
\int_0^{2\pi} r_0\sqrt[3]{r_0^2-1}\sin^2\T \d\T=\pi r_0\sqrt[3]{r_0^2-1}\neq0,
\]
for $r_0\neq 1.$
Thus, we obtain that hypothesis $ {\bf H} $ holds and Theorem \ref{thm:main3} can be applied in order to conclude this proof.\end{proof}

It is worth mentioning that the software application MATHEMATICA\textregistered~\linebreak was used to illustrate numerically the existence of the periodic solution ensured by Proposition \ref{prop:ap} for some values of $k$ and $\e.$ In Figures \ref{fig1}, \ref{fig2}, and \ref{fig3} we show the displacement function obtained for some of these simulations, which has its zero corresponding to a periodic solution.

\begin{figure}[H]
\begin{tabular}{cc}
 \includegraphics[width=6.5cm]{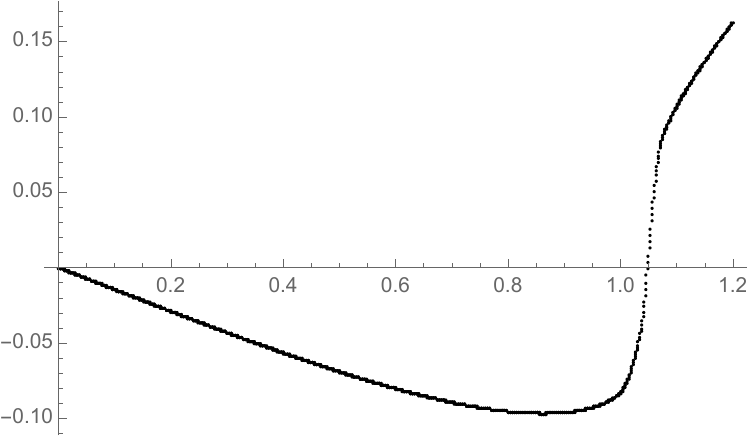} &\includegraphics[width=6.5cm]{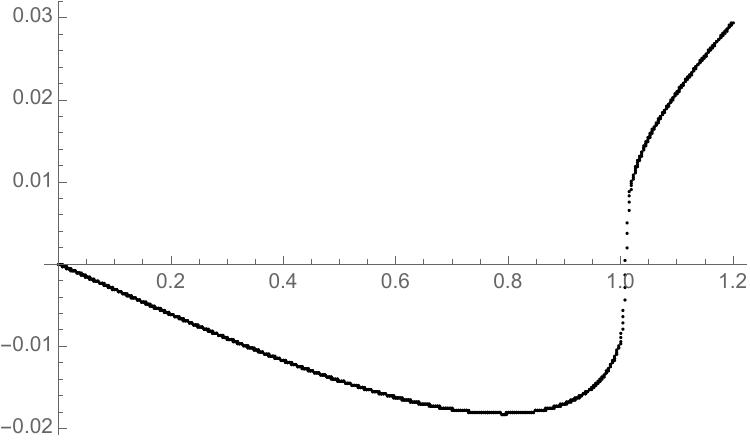}\\
\end{tabular}
\caption{Displacement function of differential equation \eqref{eq:polar} assuming $k=1$ and $E=0$ for $\e=1/20$ (left) and $\e=1/100$ (right).}\label{fig1}
\end{figure}

\begin{figure}[H]
\begin{tabular}{cc}
 \includegraphics[width=6.5cm]{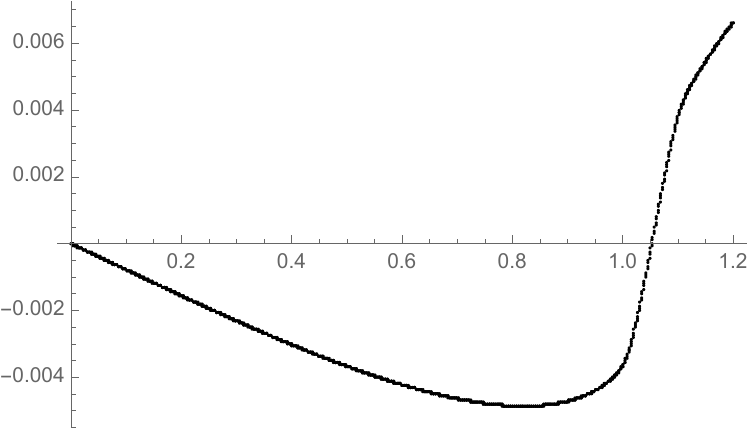} &\includegraphics[width=6.5cm]{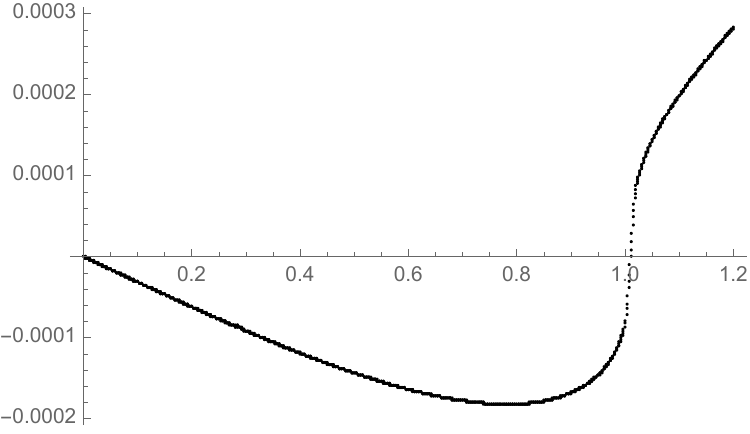}\\
\end{tabular}
\caption{Displacement function of differential equation \eqref{eq:polar} assuming $k=2$ and $E=0$ for $\e=1/20$ (left) and $\e=1/100$ (right).}\label{fig2}
\end{figure}

\begin{figure}[H]
\begin{tabular}{cc}
 \includegraphics[width=6.5cm]{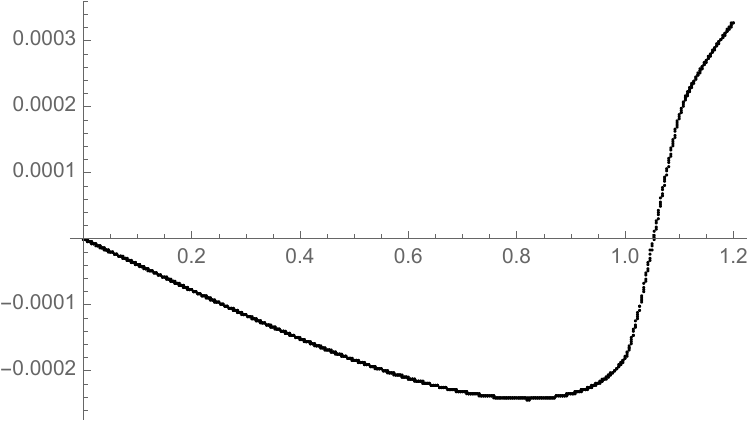} &\includegraphics[width=6.5cm]{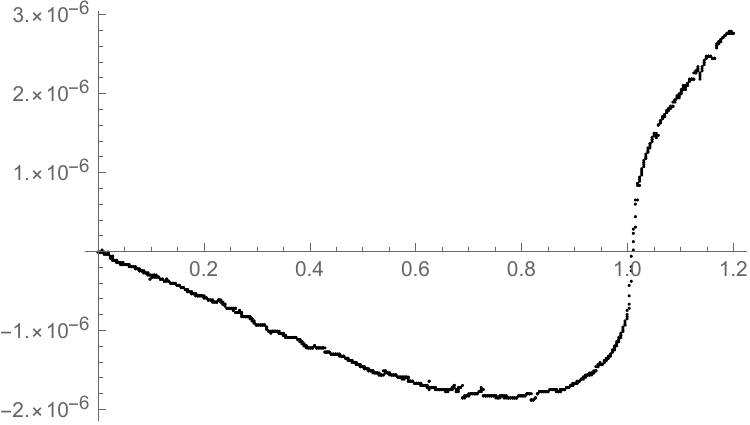}\\
\end{tabular}
\caption{Displacement function of differential equation \eqref{eq:polar} assuming $k=3$ and $E=0$ for $\e=1/20$ (left) and $\e=1/100$ (right). }\label{fig3}
\end{figure}

\section*{Acknowledgements}

The authors thank the referees for the constructive comments and suggestions which led to an improved version of the paper.

The authors also thank Espa\c{c}o da Escrita -- Pr\'{o}-Reitoria de Pesquisa -- UNICAMP for the language services provided.

DDN is partially supported by Funda\c{c}\~{a}o de Amparo \`{a} Pesquisa do Estado de S\~{a}o Paulo (FAPESP) grants 2018/16430-8, 2018/13481-0, and 2019/10269-3, and by Conselho Nacional de Desenvolvimento Científico e Tecnol\'{o}gico (CNPq) grants 306649/2018-7 and 438975/2018-9. FBGS is partially supported by Funda\c{c}\~{a}o de Amparo \`{a} Pesquisa do Estado de S\~{a}o Paulo (FAPESP) grant 2018/22689-4.

\bibliographystyle{abbrv}
\bibliography{references.bib}
\end{document}